\documentclass[12pt]{amsart}
\usepackage{amsfonts}
\usepackage{}       
\usepackage{txfonts}
\usepackage{amssymb}
\usepackage{eucal}
\usepackage{graphicx}
\usepackage{amsmath}
\usepackage{amscd}
\usepackage[all]{xy}           
\usepackage{tikz}
\usepackage{amsfonts,latexsym}
\usepackage{xspace}
\usepackage{epsfig}
\usepackage{float}
\usepackage{color}
\usepackage{fancybox}
\usepackage{colordvi}
\usepackage{multicol}
\usepackage{colordvi}
\usepackage[colorlinks,final,backref=page,hyperindex,hypertex]{hyperref}
\usepackage[active]{srcltx} 

\newtheorem{theorem}{Theorem}[section]
\newtheorem{prop}[theorem]{Proposition}
\theoremstyle{definition}
\newtheorem{defn}[theorem]{Definition}
\newtheorem{lemma}[theorem]{Lemma}

\newtheorem{prop-def}{Proposition-Definition}[section]
\newtheorem{coro-def}{Corollary-Definition}[section]

\newtheorem{exam}{Example}[section]

\newcommand{\nc}{\newcommand}
\nc{\tred}[1]{\textcolor{red}{#1}}
\nc{\tblue}[1]{\textcolor{blue}{#1}}
\nc{\tgreen}[1]{\textcolor{green}{#1}}
\nc{\tpurple}[1]{\textcolor{purple}{#1}}
\nc{\btred}[1]{\textcolor{red}{\bf #1}}
\nc{\btblue}[1]{\textcolor{blue}{\bf #1}}
\nc{\btgreen}[1]{\textcolor{green}{\bf #1}}
\nc{\btpurple}[1]{\textcolor{purple}{\bf #1}}
\nc{\NN}{{\mathbb N}}
\nc{\ncsha}{{\mbox{\cyr X}^{\mathrm NC}}} \nc{\ncshao}{{\mbox{\cyr
X}^{\mathrm NC}_0}}

\newcommand{\efootnote}[1]{}

\renewcommand{\textbf}[1]{}

\newcommand{\delete}[1]{}

\nc{\mlabel}[1]{\label{#1}}  
\nc{\mcite}[1]{\cite{#1}}  
\nc{\mref}[1]{\ref{#1}}  
\nc{\mbibitem}[1]{\bibitem{#1}} 

\delete{
\nc{\mlabel}[1]{\label{#1}  
{\hfill \hspace{1cm}{\small\tt{{\ }\hfill(#1)}}}}
\nc{\mcite}[1]{\cite{#1}{\small{\tt{{\ }(#1)}}}}  
\nc{\mref}[1]{\ref{#1}{{\tt{{\ }(#1)}}}}  
\nc{\mbibitem}[1]{\bibitem[\bf #1]{#1}} 
}

\nc{\opa}{\ast} \nc{\opb}{\odot} \nc{\op}{\bullet} \nc{\pa}{\frakL}
\nc{\arr}{\rightarrow} \nc{\lu}[1]{(#1)} \nc{\mult}{\mrm{mult}}
\nc{\diff}{\mathfrak{Diff}}
\nc{\opc}{\sharp}\nc{\opd}{\natural}
\nc{\ope}{\circ}
\nc{\dpt}{\mathrm{d}}
\nc{\tforall}{\text{ for all }}
\nc{\diam}{alternating\xspace}
\nc{\Diam}{Alternating\xspace}
\nc{\cdiam}{alternating\xspace}
\nc{\Cdiam}{Alternating\xspace}
\nc{\AW}{\mathcal{A}}
\nc{\rba}{Rota-Baxter algebra\xspace}

\nc{\ari}{\mathrm{ar}}

\nc{\lef}{\mathrm{lef}}

\nc{\Sh}{\mathrm{ST}}

\nc{\Cr}{\mathrm{Cr}}

\nc{\st}{{Schr\"oder tree}\xspace}
\nc{\sts}{{Schr\"oder trees}\xspace}

\nc{\vertset}{\Omega} 

\nc{\assop}{\quad \begin{picture}(5,5)(0,0)
\line(-1,1){10}
\put(-2.2,-2.2){$\bullet$}
\line(0,-1){10}\line(1,1){10}
\end{picture} \quad \smallskip}

\nc{\operator}{\begin{picture}(5,5)(0,0)
\line(0,-1){6}
\put(-2.6,-1.8){$\bullet$}
\line(0,1){9}
\end{picture}}

\nc{\idx}{\begin{picture}(6,6)(-3,-3)
\put(0,0){\line(0,1){6}}
\put(0,0){\line(0,-1){6}}
 \end{picture}}

\nc{\pb}{{\mathrm{pb}}}
\nc{\Lf}{{\mathrm{Lf}}}

\nc{\lft}{{left tree}\xspace}
\nc{\lfts}{{left trees}\xspace}

\nc{\fat}{{fundamental averaging tree}\xspace}

\nc{\fats}{{fundamental averaging trees}\xspace}
\nc{\avt}{\mathrm{Avt}}

\nc{\rass}{{\mathit{RAss}}}

\nc{\aass}{{\mathit{AAss}}}

\nc{\vin}{{\mathrm Vin}}    
\nc{\lin}{{\mathrm Lin}}    
\nc{\inv}{\mathrm{I}n}
\nc{\gensp}{V} 
\nc{\genbas}{\mathcal{V}} 
\nc{\bvp}{V_P}     
\nc{\gop}{{\,\omega\,}}     

\nc{\bin}[2]{ (_{\stackrel{\scs{#1}}{\scs{#2}}})}  
\nc{\binc}[2]{ \left (\!\! \begin{array}{c} \scs{#1}\\
    \scs{#2} \end{array}\!\! \right )}  
\nc{\bincc}[2]{  \left ( {\scs{#1} \atop
    \vspace{-1cm}\scs{#2}} \right )}  
\nc{\bs}{\bar{S}} \nc{\cosum}{\sqsubset} \nc{\la}{\longrightarrow}
\nc{\rar}{\rightarrow} \nc{\dar}{\downarrow} \nc{\dprod}{**}
\nc{\dap}[1]{\downarrow \rlap{$\scriptstyle{#1}$}}
\nc{\md}{\mathrm{dth}} \nc{\uap}[1]{\uparrow
\rlap{$\scriptstyle{#1}$}} \nc{\defeq}{\stackrel{\rm def}{=}}
\nc{\disp}[1]{\displaystyle{#1}} \nc{\dotcup}{\
\displaystyle{\bigcup^\bullet}\ } \nc{\gzeta}{\bar{\zeta}}
\nc{\hcm}{\ \hat{,}\ } \nc{\hts}{\hat{\otimes}}
\nc{\barot}{{\otimes}} \nc{\free}[1]{\bar{#1}}
\nc{\uni}[1]{\tilde{#1}} \nc{\hcirc}{\hat{\circ}} \nc{\lleft}{[}
\nc{\lright}{]} \nc{\lc}{\lfloor} \nc{\rc}{\rfloor}
\nc{\curlyl}{\left \{ \begin{array}{c} {} \\ {} \end{array}
    \right .  \!\!\!\!\!\!\!}
\nc{\curlyr}{ \!\!\!\!\!\!\!
    \left . \begin{array}{c} {} \\ {} \end{array}
    \right \} }
\nc{\longmid}{\left | \begin{array}{c} {} \\ {} \end{array}
    \right . \!\!\!\!\!\!\!}
\nc{\onetree}{\bullet} \nc{\ora}[1]{\stackrel{#1}{\rar}}
\nc{\ola}[1]{\stackrel{#1}{\la}}
\nc{\ot}{\otimes} \nc{\mot}{{{\boxtimes\,}}}
\nc{\otm}{\overline{\boxtimes}} \nc{\sprod}{\bullet}
\nc{\scs}[1]{\scriptstyle{#1}} \nc{\mrm}[1]{{\rm #1}}
\nc{\margin}[1]{\marginpar{\rm #1}}   
\nc{\dirlim}{\displaystyle{\lim_{\longrightarrow}}\,}
\nc{\invlim}{\displaystyle{\lim_{\longleftarrow}}\,}
\nc{\mvp}{\vspace{0.3cm}} \nc{\tk}{^{(k)}} \nc{\tp}{^\prime}
\nc{\ttp}{^{\prime\prime}} \nc{\svp}{\vspace{2cm}}
\nc{\vp}{\vspace{8cm}} \nc{\proofbegin}{\noindent{\bf Proof: }}
\nc{\proofend}{$\blacksquare$ \vspace{0.3cm}}
\nc{\modg}[1]{\!<\!\!{#1}\!\!>}
\nc{\intg}[1]{F_C(#1)} \nc{\lmodg}{\!
<\!\!} \nc{\rmodg}{\!\!>\!}
\nc{\cpi}{\widehat{\Pi}}
\nc{\sha}{{\mbox{\cyr X}}}  
\nc{\shap}{{\mbox{\cyrs X}}} 
\nc{\shan}{{\overrightarrow \sha}}
\nc{\shpr}{\diamond}    
\nc{\shp}{\ast} \nc{\shplus}{\shpr^+}
\nc{\shprc}{\shpr_c}    
\nc{\msh}{\ast} \nc{\zprod}{m_0} \nc{\oprod}{m_1}
\nc{\vep}{\varepsilon} \nc{\labs}{\mid\!} \nc{\rabs}{\!\mid}
\nc{\sqmon}[1]{\langle #1\rangle}

\nc{\mmbox}[1]{\mbox{\ #1\ }} \nc{\dep}{\mrm{dep}} \nc{\fp}{\mrm{FP}}
\nc{\rchar}{\mrm{char}} \nc{\End}{\mrm{End}} \nc{\Fil}{\mrm{Fil}}
\nc{\Mor}{Mor\xspace} \nc{\gmzvs}{gMZV\xspace}
\nc{\gmzv}{gMZV\xspace} \nc{\mzv}{MZV\xspace}
\nc{\mzvs}{MZVs\xspace} \nc{\Hom}{\mrm{Hom}} \nc{\id}{\mrm{id}}
\nc{\im}{\mrm{im}} \nc{\incl}{\mrm{incl}} \nc{\map}{\mrm{Map}}
\nc{\mchar}{\rm char} \nc{\nz}{\rm NZ} \nc{\supp}{\mathrm Supp}

\nc{\Alg}{\mathbf{Alg}} \nc{\Bax}{\mathbf{Bax}} \nc{\bff}{\mathbf f}
\nc{\bfk}{{\bf k}} \nc{\bfone}{{\bf 1}} \nc{\bfx}{\mathbf x}
\nc{\bfy}{\mathbf y}
\nc{\base}[1]{\bfone^{\otimes ({#1}+1)}} 
\nc{\Cat}{\mathbf{Cat}}

\nc{\detail}{\marginpar{\bf More detail}
    \noindent{\bf Need more detail!}
    \svp}
\nc{\Int}{\mathbf{Int}} \nc{\Mon}{\mathbf{Mon}}
\nc{\rbtm}{{shuffle }} \nc{\rbto}{{Rota-Baxter }}
\nc{\remarks}{\noindent{\bf Remarks: }} \nc{\Rings}{\mathbf{Rings}}
\nc{\Sets}{\mathbf{Sets}} \nc{\wtot}{\widetilde{\odot}}
\nc{\wast}{\widetilde{\ast}} \nc{\bodot}{\bar{\odot}}
\nc{\bast}{\bar{\ast}} \nc{\hodot}[1]{\odot^{#1}}
\nc{\hast}[1]{\ast^{#1}} \nc{\mal}{\mathcal{O}}
\nc{\tet}{\tilde{\ast}} \nc{\teot}{\tilde{\odot}}
\nc{\oex}{\overline{x}} \nc{\oey}{\overline{y}}
\nc{\oez}{\overline{z}} \nc{\oef}{\overline{f}}
\nc{\oea}{\overline{a}} \nc{\oeb}{\overline{b}}
\nc{\weast}[1]{\widetilde{\ast}^{#1}}
\nc{\weodot}[1]{\widetilde{\odot}^{#1}} \nc{\hstar}[1]{\star^{#1}}
\nc{\lae}{\langle} \nc{\rae}{\rangle}
\nc{\lf}{\lfloor}
\nc{\rf}{\rfloor}

\nc{\QQ}{{\mathbb Q}}
\nc{\RR}{{\mathbb R}} \nc{\ZZ}{{\mathbb Z}}

\nc{\cala}{{\mathcal A}} \nc{\calb}{{\mathcal B}}
\nc{\calc}{{\mathcal C}}
\nc{\cald}{{\mathcal D}} \nc{\cale}{{\mathcal E}}
\nc{\calf}{{\mathcal F}} \nc{\calg}{{\mathcal G}}
\nc{\calh}{{\mathcal H}} \nc{\cali}{{\mathcal I}}
\nc{\call}{{\mathcal L}} \nc{\calm}{{\mathcal M}}
\nc{\caln}{{\mathcal N}}\nc{\calo}{{\mathcal O}}
\nc{\calp}{{\mathcal P}} \nc{\calr}{{\mathcal R}}
\nc{\cals}{{\mathcal S}} \nc{\calt}{{\mathcal T}}
\nc{\calu}{{\mathcal U}} \nc{\calw}{{\mathcal W}} \nc{\calk}{{\mathcal K}}
\nc{\calx}{{\mathcal X}} \nc{\CA}{\mathcal{A}}

\nc{\fraka}{{\mathfrak a}} \nc{\frakA}{{\mathfrak A}}
\nc{\frakb}{{\mathfrak b}} \nc{\frakB}{{\mathfrak B}}
\nc{\frakD}{{\mathfrak D}} \nc{\frakF}{\mathfrak{F}}
\nc{\frakf}{{\mathfrak f}} \nc{\frakg}{{\mathfrak g}}
\nc{\frakH}{{\mathfrak H}} \nc{\frakL}{{\mathfrak L}}
\nc{\frakM}{{\mathfrak M}} \nc{\bfrakM}{\overline{\frakM}}
\nc{\frakm}{{\mathfrak m}} \nc{\frakP}{{\mathfrak P}}
\nc{\frakN}{{\mathfrak N}} \nc{\frakp}{{\mathfrak p}}
\nc{\frakS}{{\mathfrak S}} \nc{\frakT}{\mathfrak{T}}
\nc{\frakX}{{\mathfrak X}} \nc{\frakx}{\mathfrak{x}}

\nc{\BS}{\mathbb{S
}}

\font\cyr=wncyr10 \font\cyrs=wncyr7
\nc{\li}[1]{\textcolor{red}{#1}}
\nc{\lir}[1]{\textcolor{red}{Li:#1}}

\nc{\sz}[1]{\textcolor{blue}{SZ: #1}}

\nc{\ID}{\mathfrak{I}} \nc{\lbar}[1]{\overline{#1}}
\nc{\bre}{{\rm b}} \nc{\sd}{\cals} \nc{\rb}{\rm RB}
\nc{\A}{\rm angularly decorated\xspace} \nc{\LL}{\rm L}
\nc{\w}{\rm wid} \nc{\arro}[1]{#1}
\nc{\ver}{\rm ver}
\nc{\FN}{F_{\mathrm N}}
\nc{\FNA}{\FN(A)} \nc{\NA}{N_{A}}
\nc{\dr}{\diamond_r}
\nc{\shar}{{\mbox{\cyrs X}}_r} 
\nc{\dt}{\Delta_T}
\nc{\da}{\Delta_A}
\nc{\vt}{\vep_T }
\nc{\bul}{\bullet}
\nc{\fraku}{{\mathfrak U}}

\begin{document}

\title[Left counital Hopf algebra structures on free commutative Nijenhuis algebras ]{Left counital Hopf algebra structures on free commutative Nijenhuis algebras}
%
\author{Shanghua Zheng}
\address{Department of Mathematics, Jiangxi Normal University, Nanchang, Jiangxi 330022, China}
         \email{zheng2712801@163.com}

\author{Li Guo}
\address{Department of Mathematics and Computer Science,
         Rutgers University,
         Newark, NJ 07102, USA}
\email{liguo@rutgers.edu}

\date{\today}
\begin{abstract}
Motivated by the Hopf algebra structures established on free commutative Rota-Baxter algebras, we explore Hopf algebra related structures on free commutative Nijenhuis algebras. Applying a cocycle condition, we first prove that a free commutative Nijenhuis algebra on a left counital bialgebra (in the sense that the right-sided counicity needs not hold) can be enriched to a left counital bialgebra. We then establish a general result that a connected graded left counital bialgebra is a left counital right antipode Hopf algebra in the sense that the antipode is also only right-sided. We finally apply this result to show that the left counital bialgebra on a free commutative Nijenhuis algebra on a connected left counital bialgebra is connected and graded, hence is a left counital right antipode Hopf algebra.
\end{abstract}

\subjclass[2010]{16T99,16W99,16S10}

\keywords{Nijenhuis algebra, bialgebra, left counital bialgebra, connected bialgebra, left counital right antipode Hopf algebra}

\maketitle

\tableofcontents

\setcounter{section}{0}

\allowdisplaybreaks

\section{Introduction}

The concept of a {\bf Nijenhuis operator} first appeared in the Lie algebra context from the important notion of a Nijenhuis torsion, from the study of pseudo-complex manifolds of Nijenhuis~\mcite{N} in the 1950s. This concept is also related to the well-known concepts of Schouten-Nijenhuis bracket, the Fr\"olicher-Nijenhuis bracket~\mcite{FN} and the Nijenhuis-Richardson bracket.   Furthermore,  Nijenhuis operators  on Lie algebras play an important role  in the study
of  integrability of nonlinear evolution equations~\mcite{Do}, and  appeared in the contexts of Poisson-Nijenhuis manifolds~\mcite{KM},  the classical Yang-Baxter equation~\mcite{GS1,GS2} and  the Poisson
structure~\mcite{MM}.
More recently, Nijenhuis operators have been studied for $n$-Lie algebras, integrable systems, Hom-Lie algebra and integrable hierarchies~\mcite{C,LSZB,P,Sh}.

A Nijenhuis operator on an associative algebra $R$ is a linear endomorphism $P:R\to R$ satisfying the {\bf Nijenhuis equation}:

\begin{equation}
    P(x)P(y) = P(P(x)y) + P(xP(y)) - P^2(xy)\quad \tforall x,y \in R.
    \mlabel{eq:Nij}
\end{equation}
The Nijenhuis operator on an associative algebra was introduced by Carinena et. al.~\mcite{CGM} to study quantum bi-Hamiltonian systems.
In~\mcite{Uc}, Nijenhuis operators are constructed by analogy with Poisson-Nijenhuis geometry, from relative Rota-Baxter operators. In~\mcite{LG,Le}, relations with dendriform type algebras were studied.

The Nijenhuis operator is in close analogue with the {\bf Rota-Baxter operator} of weight $\lambda$ (where $\lambda$ is a constant), the latter being defined by the {\bf Rota-Baxter equation}
\begin{equation}
 P(x)P(y)=P(P(x)y) + P(xP(y)) +\lambda P(xy) \quad \tforall x, y\in R.
 \mlabel{eq:rbo}
 \end{equation}
Originated from the probability study of G. Baxter~\mcite{Ba}, the Rota-Baxter operator was studied by Cartier and Rota~\mcite{Ca,Ro} and is closely related to the operator form of the classical Yang-Baxter equation~\mcite{Bai,STS}. Its intensive study in the last two decades found many applications in mathematics and physics, most notably the work of Connes and Kreimer on renormalization of quantum field theory~\mcite{CK,EGK,EGM}. See~\mcite{Gub} for further details and references.

Theoretic developments of Nijenhuis algebras have been in parallel to those of Rota-Baxter algebras. For example, free commutative and noncommutative Nijenhuis algebras were constructed in~\mcite{EL,LG} following the construction of free commutative and noncommutative Rota-Baxter algebras~\mcite{GK1,EG}. In~\mcite{ZGGS}, the Nijenhuis operator, as one of the Rota-Baxter type operators, is studied in the context of Rota's problem~\mcite{Ro2} on classification of operator identities.

Motivated by the close relationship between divided power Hopf algebra and shuffle product Hopf algebra with Rota-Baxter algebras, Hopf algebra structures have been found on free commutative Rota-Baxter algebras for quite a few years~\mcite{AGKO,EG1}. Recently, Hopf algebra structures on free (non-commutative) Rota-Baxter algebras have been obtained~\mcite{ZGG} by a one-cocycle property in analogous to the Connes-Kreimer Hopf algebra of rooted trees~\mcite{CK}. Thus it is natural to explore a Hopf algebra structure on free Nijenhuis algebras. The purpose of this paper is to pursue this for free commutative Nijenhuis algebras. As it turns out, this approach does not give a bona fide Hopf algebra, only one with a left-sided counit and right-sided antipode. Interestingly, related structures have appeared in the study of quantum groups~\mcite{RT} (tracing back to~\mcite{GNT}) and combinatorics~\mcite{FLS,FS}. In both cases, the term one-side Hopf algebra was used though in the former case only one-sided antipode is optional while in the latter case both a one-sided antipode and the opposite side unit are optional. To distinguish from these notions, we will use the terms left counital bialgebra and left counital Hopf algebra for our constructions. See~\mcite{BNS,Li,Wa} for some other variations of Hopf algebras with relaxed conditions.

The layout of the paper is as follows. In Section~\mref{sec:FreeCNij} we recall the construction of free commutative Nijenhuis algebras by a generalization of the shuffle product and discuss a combinatorial identity from a special case. In Section~\mref{sec:bialg} we equip the free commutative Nijenhuis algebras with a left counital bialgebra. In Section~\mref{sec:hopf}, we first extend the classical result that a connected graded bialgebra is a Hopf algebra to the left unital context. We then show that the above-mentioned left counital bialgebra structure on free commutative Nijenhuis algebras is graded and connected, and thus yields a left counital Hopf algebra.

\smallskip

\noindent
{\bf Convention. } In this paper, all algebras are taken to be unitary commutative over a unitary commutative ring $\bfk$. Also linear maps and tensor products are taken over $\bfk$.

\section{Free commutative Nijenhuis algebras}
\label{sec:FreeCNij}

In this section, we recall the notion of Nijenhuis algebras and the construction of free commutative Nijenhuis algebras by the right-shift shuffle product~\mcite{EL}.
We then consider some special cases of free commutative Nijenhuis algebras.

\subsection{Free commutative Nijenhuis algebras on a commutative algebra}

\begin{defn}
A {\bf Nijenhuis algebra} is an associative algebra $R$ equipped with a linear operator $P$, called {\bf Nijenhuis operator}, satisfying the {\bf Nijenhuis equation} in Eq.~(\mref{eq:Nij}). A homomorphism from a Nijenhuis algebra $(R,P)$ to a Nijenhuis algebra $(S,Q)$ is an algebra homomorphism $f: R\to S$ such that $f P = Qf$.
\end{defn}

We recall the concept and construction of the free commutative Nijenhuis algebra on a commutative algebra.

\begin{defn}
{\rm
Let $A$ be a commutative algebra. A free commutative Nijenhuis
algebra on $A$ is a commutative Nijenhuis algebra $\FN(A)$ with a
Nijenhuis operator $\NA$ and an algebra homomorphism $j_A:
A\to \FN(A)$ such that, for any commutative Nijenhuis algebra $(R,P)$ and
any  algebra homomorphism $f:A\to R$, there is a unique
 Nijenhuis algebra homomorphism $\free{f}: \FN(A)\to R$
such that $f=\free{f}\circ j_A$, that is, the following diagram
$$ \xymatrix{ A \ar[rr]^{j_A}\ar[drr]^{f} && \FN(A) \ar[d]_{\free{f}} \\
&& R}
$$
commutes.}
\mlabel{de:CNij}
\end{defn}

For a given unital commutative algebra $A$ with unit $1_A$, we will give a Nijenhuis algebra structure on the $\bfk$-module
\begin{equation}
\shan(A):=\bigoplus_{n\geq 1}A^{\ot n}. \notag
\end{equation}
Here $A^{\ot n}$ is the $n$-th tensor power of $A$.

We first define the {\bf right-shift operator $P_r$} on $\shan (A)$ by
\begin{equation}
P_r:\shan (A)\to \shan (A), \quad P_r(\fraka)=1_A\ot \fraka\ \  \tforall \fraka\in A^{\ot n}, n\geq 1.
\end{equation}
We then define a multiplication on $\shan (A)$ as follows.

For $\fraka=a_1\ot \cdots \ot a_m\in A^{\ot m}$ and $\frakb=b_1\ot \cdots \ot b_n\in A^{\ot n}$, denote $\fraka'=a_2\ot \cdots \ot a_m$ if $m\geq 2$ and $\frakb'=b_2\ot \cdots \ot b_n$ if $n\geq 2$, so that $\fraka =a_1\ot \fraka'$ and $\frakb=b_1\ot \frakb'$. Define a multiplication $\dr$ on $\shan(A)$ by the following recursion.
\begin{equation}
\fraka\dr \frakb = \left \{\begin{array}{ll} a_1b_1, & m=n=1, \\
a_1b_1\ot \frakb', & m=1, n\geq 2, \\
a_1b_1\ot \fraka', & m\geq 2, n=1,\\
a_1b_1\ot \Big( \fraka' \dr (1\ot \frakb') + (1\ot \fraka')\dr \frakb' - 1\ot (\fraka' \dr \frakb')\Big), & m, n\geq 2.
\end{array} \right .
\label{eq:dfndr}
\end{equation}

Alternatively, $\dr$ can be defined by
$$ \fraka\dr \frakb = a_1b_1\ot (\fraka' \shar \frakb')$$
where $\shar$ is the right-shift shuffle product defined in~\mcite{EL} whose precise definition we will not recall since it is not needed in the sequel.

Let
$$j_A: A\to \shan (A),\,a\mapsto a,$$
be the natural embedding. Then
$$j_A(ab)=ab=a\dr b=j_A(a)\dr j_A(b),\,\tforall\, a,b\in A.$$
So $j_A$ is an algebra homomorphism.
Then we have
\begin{theorem}{\bf \cite{EL}}
Let $A$ be a commutative algebra with unit $1_A$. Let $\shan (A),P_r$, $\dr$ and $j_A$ be defined as above.  Then
\begin{enumerate}
\item
The triple $(\shan (A), \dr, P_r)$ is a commutative Nijenhuis algebra;
\item
The quadruple~$(\shan (A),\dr,P_r, j_A)$ is the free commutative Nijenhuis algebra on $A$.
\end{enumerate}
\mlabel{thm:freeCNij}
\end{theorem}

\subsection{Special cases} Let $X$ be a nonempty set and let $A=\bfk[X]$. Then the Nijenhuis algebra $\shan (\bfk[X])$ is the free commutative Nijenhuis algebra on $X$ defined by the usual universal property.

Now let $A=\bfk$. Then we obtain the free commutative Nijenhuis algebra
$$\shan (\bfk)=\bigoplus_{n\geq 0}\bfk^{\ot (n+1)}.$$
We denote the unit $1_\bfk$ of $\bfk$ by $1$ for abbreviation. Since the tensor product is over $\bfk$, by bilinearity, we have $\bfk^{\ot (n+1)}=\bfk 1^{\ot (n+1)}$ for all $n\geq0$.
Then
\begin{equation}
\shan (\bfk)=\bigoplus_{n\geq 0} \bfk 1^{\ot (n+1)}.
\mlabel{eq:dirsum}
\end{equation}
Thus $\shan (\bfk)$ is a free $\bfk$-module on the basis $1^{\ot n}$, $n\geq 1$.
\begin{prop} For any $m,n\in\NN$,
\begin{equation}
1^{\ot (m+1)}\dr1^{\ot (n+1)}=1^{\ot (m+n+1)}.
\mlabel{eq:spec}
\end{equation}
\mlabel{prop:spec}
\end{prop}
\begin{proof}
We prove Eq.~\eqref{eq:dirsum} by  induction on $m+n\geq 0$. For $m+n=0$, we get $m=n=0$. Then the definition of $\dr$ in Eq.~(\mref{eq:dfndr}) yields
$1 \dr 1 = 1$.
For $k\geq 0$, assume that Eq.~\eqref{eq:dirsum} has been proved for $m+n\geq k$. Consider the case of $m+n=k+1$. Then either $m\geq 1$ or $n\geq 1$. If one of $m$ or $n$ is zero, then by Eq.~\eqref{eq:dfndr}, we have
$1^{\ot (m+1)}\dr 1^{\ot (n+1)}=1^{\ot (m+n+1)}.$ If none of $m$ or $n$ is zero, then by Eq.~\eqref{eq:dfndr} and the induction hypothesis, we have
\begin{eqnarray*}
1^{\ot (m+1)}\dr 1^{\ot (n+1)}&=& 1\ot \left( 1^{\ot m}\dr 1^{\ot (n+1)}+1^{\ot (m+1)}\dr 1^{\ot n} - 1\ot (1^{\ot m}\dr 1^{\ot n})\right)\\
&=& 1\ot \left (1^{\ot (m+n)}+1^{\ot (m+n)}-1\ot 1^{\ot (m+n-1)}\right)\\
&=& 1^{\ot (m+n+1)}.
\end{eqnarray*}
This completes the induction.
\end{proof}

Note that the definition of the Nijenhuis operator in Eq.~\eqref{eq:Nij} can be regarded as formally taking $\lambda=-P$ in the definition of Rota-Baxter operator in Eq.~\eqref{eq:rbo}. In fact, the construction of commutative Nijenhuis algebras is similar to that of free commutative Rota-Baxter algebras $(\sha_\bfk(A):=\sha (A),P_A,\diamond_\lambda)$ on $A$ of weight $\lambda\in\bfk$~\cite[Chapter~3, Theorem~3.2.1]{Gub}. Let $A=\bfk$. Then $(\sha_\bfk(\bfk),P_\bfk,\diamond_\lambda)$ is the free commutative Rota-Baxter algebra on~$\bfk$ of weight $\lambda$. When $A=\bfk$, the module structure of $\sha(A)$ is
$$\sha_\bfk(\bfk)=\bigoplus_{n\geq0}\bfk1^{\ot(n+1)}$$
and the multiplication $\diamond_\lambda$ is given by

\begin{equation}
1^{\ot (m+1)}\diamond_\lambda 1^{\ot (n+1)}=\sum_{k=0}^{\min\{m,n\}}{m+n-k\choose m}{m\choose k}\lambda^k 1^{\ot(m+n+1-k)} \ \ \tforall m, n\geq0.
\mlabel{eq:rbeq}
\end{equation}

Since the Rota-Baxter operator can be transformed to the Nijenhuis operator by formally replacing $\lambda$ by $-P$, doing so and hence $\lambda^k$ by $(-1)^k P^k$ in the above equation, we can expect to obtain the corresponding formula for the free Nijenhuis algebra $\shan(\bfk)$, namely,

$$ 1^{\ot (m+1)}\dr 1^{\ot (n+1)}=\sum_{k=0}^{\min\{m,n\}}{m+n-k\choose m}{m\choose k}(-1)^k P_r^k ( 1^{\ot(m+n+1-k)})
=\sum_{k=0}^{\min\{m,n\}}{m+n-k\choose m}{m\choose k}(-1)^k 1^{\ot(m+n+1)}$$
since $P_r(1^{\ot \ell})=1^{\ot (\ell+1)}.$
Comparing with Eq.~\eqref{eq:spec} suggests the identity

\begin{equation}
\sum_{k=0}^{\min\{m,n\}}(-1)^k{{m+n-k}\choose{m}}{{m}\choose{k}}\,=\,1.
\notag
\end{equation}
This is in fact a combinatorial identity which is related to the probability distribution in quantum field theory. See \cite{Lee} where the formula, under the assumption $m\geq n$ without loss of generality, is proved by identifying the coefficient of $x^m$ in the series expansion of the product
$$ (1-x)^n(1-x)^{-n-1}=(1-x)^{-1}=\sum_{k=0}^\infty x^k.$$
The formula can also be proved by enumerating stuffles~\cite{Gub}.

We end this section by equipping $\shan(\bfk)$ with a natural Hopf algebra structure.
Denote $u_n:=1^{\ot (n+1)},n\geq 0$. Then by Eq.~(\mref{eq:spec}),
\begin{equation}
u_n\dr u_m=u_{m+n},\,\tforall\, m,n\geq 0.
\end{equation}
Thus $(\shan (\bfk),\dr)$ is the free commutative algebra (that is, the polynomial algebra) $\bfk[u_1]$ generated by $u_1(=1^{\ot 2})$. Using the universal property of free commutative algebras, we find that $\shan (\bfk)$ has a unique bialgebra structure with the comultiplication $\Delta(u_n)=\sum_{i=0}^n{n\choose i}u_i\ot u_{n-i}$ and the counit $\vep(u_0)=1$, and $\vep(u_n)=0, \,n\geq 1$. Define a linear map $u:\bfk\to \shan (\bfk), c\mapsto c$ for all $c\in \bfk$. Thus we have
\begin{theorem}
\begin{enumerate}
\item
The bialgebra $(\shan (\bfk),\dr,u,\Delta,\vep)$  with the unique antipode
$$S:\shan (\bfk)\to \shan (\bfk),\,\,u_n\mapsto (-1)^nu_n,\,n\geq 0,$$
is the binomial Hopf algebra.
\item
The free commutative Nijenhuis algebra $\shan (\bfk)$ is a Hopf algebra.
\end{enumerate}
\mlabel{thm:bfk}
\end{theorem}
We will consider its generalization to other free Nijenhuis algebras in the next section.

\section{The left counital bialgebra structure on free commutative Nijhenhuis algebras}
\mlabel{sec:bialg}

In this section, we will equip a free commutative Nijenhuis algebra $\shan(A)$ with a left conunital Hopf algebra structure, when the generating algebra $A$ is a left counital bialgebra. Let $A:=(A,m_A,\mu_A,\Delta_A,\vep_A)$ be a left counital bialgebra. We first construct a comultiplication on the free commutative Nijehuis algebra $\shan (A):=(\shan (A), \dr, P_r)$ by applying a one-cocycle property. Then we give the construction of a left counit on $\shan (A)$. The left counital Hopf algebra property of $\shan(A)$ will be considered in the next section.

\subsection{Comultiplication by cocycle condition}
First we give the construction of left counital coalgebra on $\shan (A)$ by giving the comultiplication and the left counit on free commutative  Nijenhuis algebras. The coassociativity and left counicity will be proved later.

First we introduce the notions of an operated bialgebra and a cocycle bialgebra with a left counit.
\begin{defn}
\begin{enumerate}
\item
A {\bf left counital coalgebra} is a triple $(H,\Delta,\vep)$, where the comultiplication $\Delta:H\to H\ot H$ satisfies the coassociativity and the counit $\vep:H\to \bfk$ satisfies the {\bf left  counicity}: $(\vep\ot\id)\Delta =\beta_\ell$, where $\beta_\ell:H\to \bfk\ot H, u\mapsto 1\ot u\ \tforall u\in H$;
\item
A {\bf left counital operated bialgebra} is a sextuple $(H,m,\mu,\Delta,\vep,P)$, where the triple $(H,\Delta,\vep)$ is a left counitial coalgebra, and the pair $(H,P)$ is an operated algebra, that is, an algebra $H$ with a linear operator $P:H\to H$;
\item
A {\bf left  counital cocycle bialgebra} is a left  operated bialgebra $(H,m,\mu,\Delta,\vep,P)$ satisfying the one-cocycle property\footnote{The cocycle condition $\Delta P=P\ot 1+(\id \ot P)\Delta$, which is use to construct the Hopf algebra structure on free Rota-Baxter algebras~\cite{ZGG}, does not work for free Nijenhuis algebras.}:
\begin{equation}
\Delta P=(\id\ot P)\Delta.
\mlabel{eq:cocy}
\end{equation}
\end{enumerate}
\end{defn}

Any bialgebra is a left counital bialgebra. As a simple example of left counital bialgebra which is not a bialgebra, consider $\bfk[x]$ with the coproduct
$\Delta(u)=1\ot u$ for all $u\in \bfk[x]$ and the usual counit
$\vep:\bfk[x]\to \bfk$ defined by $\vep(x)=0, \vep(1)=1$. Then it is easy to check all the conditions of a bialgebra except the right counicity: $(\id\ot \vep)\Delta = \beta_r$ where
$\beta_r:\bfk[x]\to \bfk[x]\ot \bfk, u\mapsto u\ot 1$.
Since
$$(\id\ot \vep)\Delta (x)=(\id \ot \vep)(1\ot x)=1\ot 0=0\neq x\ot 1,$$
the right counicity does not hold.

Note that the tensor product $\shan (A)\ot \shan (A)$ is also an associative algebra whose multiplication will be denoted by $\bul$.

Now we give the definition of the comultiplicaiton $\dt:\shan (A)\to \shan (A)\ot \shan (A)$. For  every pure tensor
$\fraka:=a_1\ot a_2\ot\cdots\ot a_n\in A^{\ot n},n\geq 1$, we define $ \dt$ by induction on $n\geq 1$. For $n=1$, that is, $\fraka=a_1\in A$, we define $\dt(\fraka):=\da(a_1)$ to be the coproduct $\da$ on $A$.
Assume that $\dt$ has been defined for $n\geq 1$. Consider
 $\fraka\in A^{\ot (n+1)}$. Then $\fraka=a_1\ot\fraka'$ with $\fraka':=a_2\ot \cdots\ot a_n \in A^{\ot n}.$
 By Eq.~(\mref{eq:dfndr}), we have
\begin{equation}
a_1\ot\fraka'=a_1\dr(1_A\ot \fraka')=a_1\dr P_r(\fraka').
\mlabel{eq:drbase}
\end{equation} We first define
 \begin{equation}
 \dt(1_A\ot \fraka')=(\id\ot P_r)\dt(\fraka').
 \mlabel{eq:dtpr}
 \end{equation}
We then  define
\begin{equation}
\dt(a_1\ot \fraka')=\Delta_A(a_1)\bul\dt(1_A\ot \fraka')(=\Delta_A(a_1)\bul(\id\ot P_r)\dt(\fraka')).
\mlabel{eq:dfndtre}
\end{equation}
Here since $\fraka'$ is in $A^{\ot n}$, $\dt(\fraka')$ in Eq.~(\mref{eq:dtpr}) is well-defined by the induction hypothesis. Thus $\dt(\fraka)$ is well-defined.
The definition of $\dt$ also can be rewritten as follows.
\begin{equation}
\dt(\fraka)=
\da(a_1)\bul\Bigg((\id\ot P_r)\Bigg(\da(a_2)\bul\Big((\id\ot P_r)\big(\cdots\bul(\id\ot P_r)\da(a_n)\big)\Big)\Bigg)\Bigg).
\mlabel{eq:dfndt}
\end{equation}

\begin{exam}Note that $\bfk$-algebra $\bfk$  has a natural bialgebra structure, where the comulitplication $\Delta_\bfk$ and the counit $\vep_\bfk$ are defined by
\begin{equation}
\Delta_\bfk:\bfk\to \bfk\ot \bfk,\,c\mapsto c\ot 1,\quad\text{and}\quad \vep_\bfk:\bfk\to \bfk, c\mapsto c.\,\tforall\,c\in \bfk.
\end{equation}
Then $\bfk$ becomes a left counital bialgebra. Then by Eq.~(\mref{eq:dirsum}), we get
 $$\shan (\bfk)=\bigoplus_{n\geq 0} \bfk 1^{\ot (n+1)}.$$
For every $c_i\in\bfk,1\leq i\leq 3$, by the definition of $\dt$ in Eq.~(\mref{eq:dfndt}),
 \begin{eqnarray*}
 \dt(c_1\ot c_2\ot c_3)&=&\Delta_\bfk(c_1)\bul\Bigg((\id\ot P_r)\Bigg(\Delta_\bfk(c_2)\bul\Big((\id\ot P_r)\Delta_\bfk(c_3)\Big)\Bigg)\Bigg)\\
 &=&\Delta_\bfk(c_1)\bul\Bigg((\id\ot P_r)\Bigg((c_2\ot 1)\bul\Big(c_3\ot 1^{\ot 2}\Big)\Bigg)\Bigg)\\
 &=&\Delta_\bfk(c_1)\bul\Bigg((\id\ot P_r)\Bigg(c_2c_3\ot 1^{\ot 2}\Big)\Bigg)\Bigg)\quad(\text{by Eq.~(\mref{eq:dfndr}}))\\
 &=&(c_1\ot 1)\bul(c_2c_3\ot 1^{\ot 3})\quad(\text{by Eq.~(\mref{eq:dfndr}}))\\
 &=&c_1c_2c_3\ot1^{\ot 3}.
 \end{eqnarray*}
 More generally, let $u_n:=1^{\ot (n+1)}$ for all $n\geq 0$. Then
 $$\dt(u_0)=\Delta_\bfk(1)=1^{\ot 2}=1\ot u_0,$$
By Eq.~(\mref{eq:dtpr}), we obtain
 $$\dt(u_n)=(\id\ot P_r)\dt(u_{n-1})=1\ot u_n.$$
 Thus $\dt$ is different from the coproduct defined in Theorem~\mref{thm:bfk}.
\end{exam}

We next give the construction of the left counit $\vt$ on $\shan (A)$ by using the left counit $\vep_A$ of $A$.
For every pure tensor $\fraka=a_1\ot a_2\ot \cdots\ot a_n\in A^{\ot n}$ with $n\geq 1$, we define
\begin{equation}
\vt:\shan (A)\to \bfk,\,\fraka\mapsto\vep_A(a_1)\vep_A(a_2)\cdots \vep_A(a_n).
\mlabel{eq:dfnvt}
\end{equation}

\begin{lemma} For any pure tensors $\fraka\in A^{\ot m}$ and $\frakb\in A^{\ot n}$ with $m,n\geq1$, we have
\begin{equation}
(\id\ot \dt)(\id\ot P_r)(\fraka\ot \frakb)=(\id\ot \id \ot P_r)(\id \ot \dt)(\fraka\ot \frakb)
\mlabel{eq:comupr1}
\end{equation}
and
\begin{equation}
(\dt\ot \id)(\id\ot P_r)(\fraka\ot \frakb)=(\id\ot \id \ot P_r)(\dt \ot\id)(\fraka\ot \frakb).
\mlabel{eq:comupr2}
\end{equation}
\mlabel{lem:comupr}
\end{lemma}
\begin{proof} We first prove that Eq.~(\mref{eq:comupr1}) holds for any pure tensors $\fraka\in A^{\ot m}$ and $\frakb\in A^{\ot n}$. By Eq.~(\mref{eq:dtpr}), we obtain
\begin{eqnarray*}
(\id\ot \dt)(\id\ot P_r)(\fraka\ot \frakb)&=&(\id\ot \dt P_r)(\fraka\ot\frakb)\\
&=&\fraka \ot (\dt P_r(\frakb))\\
&=&\fraka\ot ((\id \ot P_r)\dt (\frakb))\\
&=&(\id\ot \id\ot P_r)(\id \ot \dt)(\fraka\ot \frakb).\\
\end{eqnarray*}
 Eq.~(\mref{eq:comupr2}) follows from
\begin{eqnarray*}
(\dt\ot \id)(\id\ot P_r)(\fraka\ot \frakb)&=& (\dt\ot P_r)(\fraka\ot \frakb)\\
&=&\dt(\fraka)\ot P_r(\frakb)\\
&=&(\id\ot\id\ot P_r)(\dt\ot\id)(\fraka\ot\frakb).
\end{eqnarray*}
\end{proof}

\subsection{The compatibilities of $\dt$ and $\vt$}
\mlabel{sec:Compability}
We now prove that $\dt$  and $\vt$ are algebra homomorphisms.

\begin{prop}
The comultiplication $\dt:\shan (A)\to \shan (A)\ot \shan (A)$ is an algebra homomorphism.
\mlabel{prop:comuhom}
\end{prop}

\begin{proof}To prove that $\dt$ is an algebra homomorphism,
we only need to verify
\begin{equation}
\dt(\fraka\dr\frakb)=\dt(\fraka)\bul\dt(\frakb)
\mlabel{eq:alghom}
\end{equation}
for any pure tensors $\fraka:=a_1\ot \cdots \ot a_m\in A^{\ot m}$ and $\frakb:=b_1\ot \cdots \ot b_n, \in A^{\ot n},m,n\geq 1$. We do this by applying the induction on $m+n\geq 2$. If $m+n=2$, then $m=n=1$. So $\fraka,\frakb$ are in $A$ and Eq.~\eqref{eq:alghom} follows from the assumption that $\da:A\to A\ot A$ is an algebra homomorphism.

Let $k\geq 2$. Assume that Eq.~(\mref{eq:alghom}) has been proved for the case of $m+n\leq k$. Consider the case of $m+n=k+1$. Then either $m\geq 2$ or $n\geq 2$. We first verify the case when $m\geq 2$ and $n\geq 2$. Write $\fraka=a_1\ot \fraka'$ with $\fraka'=a_2\ot\cdots \ot a_m\in A^{\ot (m-1)}$ and $\frakb=b_1\ot \frakb'$ with $\frakb'=b_2\ot \cdots \ot b_n\in A^{\ot (n-1)}$. Then the left hand side of Eq.~(\mref{eq:alghom}) is

\begin{eqnarray*}
\dt(\fraka\dr\frakb)&=&\dt((a_1\ot \fraka')\dr (b_1\ot \frakb'))\\
&=&\dt((a_1\dr P_r(\fraka'))\dr(b_1\dr P_r(\frakb')))\\
&=&\dt((a_1b_1)\dr(P_r(\fraka')\dr P_r(\frakb')))\quad(\text{by the commutativity of}~\dr)\\
&=&\dt((a_1b_1)\dr\bigg(P_r\Big(\fraka'\dr P_r(\frakb')+P_r(\fraka')\dr\frakb'-P_r(\fraka'\dr\frakb')\Big)\bigg)\quad(\text{by Eq.~\eqref{eq:Nij}})\\
&=&\dt(a_1b_1\ot(\fraka'\dr(1_A\ot\frakb')+(1_A\ot\fraka')\dr\frakb'-1_A\ot (\fraka'\dr\frakb'))\quad(\text{by Eq.}~(\mref{eq:dfndr}))\\
&=& \da(a_1b_1)\bul\bigg((\id\ot P_r)\dt(\fraka'\dr(1_A\ot\frakb'))+(\id\ot P_r)\dt((1_A\ot\fraka')\dr\frakb')\\
&&-(\id\ot P_r)\Big((\id\ot P_r)\dt(\fraka'\dr\frakb')\Big)\bigg)\quad(\text{by Eqs.}~(\mref{eq:dtpr}) \text{ and} ~(\mref{eq:dfndtre})) \\
&=&\da(a_1b_1)\bul\bigg((\id\ot P_r)\Big(\dt(\fraka')\bul\dt(1_A\ot\frakb')\Big)+(\id\ot P_r)\Big(\dt(1_A\ot\fraka')\bul\dt(\frakb')\Big)\\
&&-(\id\ot P_r)\Big((\id\ot P_r)\big(\dt(\fraka')\bul\dt(\frakb')\big)\Big)\bigg)\quad(\text{by the induction hypothesis})\\
&=&\da(a_1b_1)\bul\bigg((\id\ot P_r)\Big(\dt(\fraka')\bul(\id\ot P_r)\dt(\frakb')\Big)\\
&&+(\id\ot P_r)\Big((\id\ot P_r)\dt(\fraka')\bul\dt(\frakb')\Big)\\
&&-(\id\ot P_r)\Big((\id\ot P_r)\big(\dt(\fraka')\bul\dt(\frakb')\big)\Big)\bigg) \quad(\text{by Eq.}~(\mref{eq:dtpr})).
\end{eqnarray*}
The right hand side of Eq.~(\mref{eq:alghom}) is
\begin{eqnarray*}
\dt(\fraka)\bul\dt(\frakb)&=&\dt(a_1\ot\fraka')\bul\dt(b_1\ot \frakb')\\
&=&\Big(\da(a_1)\bul (\id\ot P_r)\dt(\fraka')\Big)\bul \Big((\da(b_1)\bul(\id\ot P_r)\dt(\frakb')\Big)\quad(\text{by Eq.}~(\mref{eq:dfndtre}))\\
&=&\da(a_1b_1)\bul\Big((\id\ot P_r)\dt(\fraka') \bul(\id\ot P_r)\dt(\frakb')\Big)\quad(\text{by the commutativity of }~\bul)\\
&=&\da(a_1b_1)\bul\bigg((\id\ot P_r)\Big(\dt(\fraka')\bul (\id\ot P_r)\dt(\frakb')\Big)\\
&&+(\id\ot P_r)\Big((\id\ot P_r)\dt(\fraka')\bul\dt(\frakb')\Big)\\
&&-(\id\ot P_r)\Big((\id\ot P_r)\big(\dt(\fraka')\bul\dt(\frakb')\big)\Big)\bigg)\quad(\text{by Eq.~\eqref{eq:Nij}}).
\end{eqnarray*}
Thus $\dt(\fraka\dr\frakb)=\dt(\fraka)\bul\dt(\frakb).$ The verification of the other cases, namely when one of $m$ or $n$ is one, is simpler. This completes the induction.
\end{proof}
Now we prove that the counit $\vt:\shan (A)\to \bfk$ defined in Eq.~(\mref{eq:dfnvt} is an algebra homomorphism.
\begin{prop}
The counit $\vt$ is an algebra homomorphism.
\mlabel{prop:counithom}
\end{prop}
\begin{proof}
It suffices to prove that
 \begin{equation}
 \vt(\fraka \dr\frakb)=\vt(\fraka)\vt(\frakb)
 \mlabel{eq:comvt}
 \end{equation}
 for any pure tensors $\fraka:=a_1\ot a_2\ot \cdots\ot a_m\in A^{\ot m}$ and
 $\frakb:=b_1\ot b_2\ot \cdots \ot b_n\in A^{\ot n}$  with $m,n\geq 1$.
We proceed with induction on $m+n\geq 2$. If $m=n=1$, then by Eqs.~(\mref{eq:dfndr}) and ~(\mref{eq:dfnvt})), we obtain
$$
\vt(\fraka\dr\frakb)=\vep_A(\fraka\frakb)
=\vep_A(\fraka)\vep_A(\frakb)
=\vt(\fraka)\vt(\frakb).
$$
Let $k\geq 2$. Assume that Eq.~(\mref{eq:comvt}) holds for $m+n\leq k$ and consider $m+n=k+1$. Then either $m\geq 2$ or $n\geq 2$. We will only verify the case when both $m\geq 2$ and $n\geq 2$ since the other cases are simpler. By Eq.~(\mref{eq:dfnvt}), we have
\begin{equation}
\vt(a_1\ot\fraka')=\vep_A(a_1)\vt(\fraka')=\vt(a_1)\vt(\fraka').
\mlabel{eq:vpr}
\end{equation}
Especially, $\vt(P_r(\fraka))=\vt(\fraka)$ since $\vep_A(1_A)=1$.
 Then
 \begin{eqnarray*}
 &&\vt(\fraka\dr\frakb)\\
 &=&\vt((a_1\ot \fraka')\dr (b_1\ot \frakb'))\\
 &=&\vt((a_1\dr P_r(\fraka')\dr (b_1\dr P_r(\frakb')\quad(\text{by Eq.}~(\mref{eq:drbase}))\\
 &=&\vt((a_1b_1)\dr P_r(\fraka'\dr P_r(\frakb')+P_r(\fraka')\dr\frakb'-P_r(\fraka'\dr \frakb'))\quad(\text{by Eq.~\eqref{eq:Nij}})\\
 &=&\vt((a_1b_1)\ot (\fraka'\dr P_r(\frakb')+P_r(\fraka')\dr\frakb'-P_r(\fraka'\dr \frakb')\quad(\text{by Eq.}~(\mref{eq:dfndr}))\\
 &=&\vt(a_1b_1)\bigg(\vt((\fraka'\dr P_r(\frakb'))+\vt(P_r(\fraka')\dr\frakb')-\vt(P_r(\fraka'\dr \frakb'))\bigg)\quad(\text{by Eq.}~(\mref{eq:dfnvt}))\\
 &=&\vt(a_1b_1)\bigg(\vt(\fraka')\vt(\frakb') +\vt(\fraka')\vt(\frakb')-\vt(\fraka')\vt(\frakb')\bigg) \\
  && \qquad \qquad (\text{by the induction hypothesis and Eq.}~(\mref{eq:vpr}))\\
 &=&\vt(a_1)\vt(b_1)\vt(\fraka')\vt(\frakb')\\
 &=&\vt(a_1\ot \fraka')\vt(b_1\ot \frakb')\quad(\text{by Eq.}~(\mref{eq:vpr}))\\
 &=&\vt(\fraka)\vt(\frakb).
 \end{eqnarray*}
\end{proof}

\subsection{The coassociativity of $\dt$ and the left counicity of $\vt$}
\mlabel{sec:Coasso}

We will prove that $\dt$ is coassociative and $\vt$ satisfies the left counicity property.

\begin{prop}
The comultiplication $\dt$ is coassociative, that is,
\begin{equation}
(\id\ot \dt)\dt=(\dt\ot\id )\dt.
\mlabel{eq:comult}
\end{equation}
\mlabel{prop:comult}
\end{prop}

\begin{proof}
To prove Eq.~(\mref{eq:comult}), we only need to verify the equation
\begin{equation}
(\id\ot \dt)\dt(\fraka)=(\dt\ot\id )\dt(\fraka)
\mlabel{eq:cop}
\end{equation}
for any pure tensor $\fraka:=a_1\ot \fraka'\in A^{\ot n}$ with $n\geq 1$. For this we apply the induction on $n\geq 1$.
If $n=1$, then $\fraka=a_1$ is in $A$, so Eq.~\eqref{eq:cop} follows from the coassociativity of $\da$.

Let $k\geq 1$.
Assume that Eq.~(\mref{eq:cop}) holds for $\fraka\in A^{\ot k}$. Consider $\fraka=a_1\ot \fraka'\in A^{\ot (k+1)}$. Then we have
\begin{eqnarray*}
(\id\ot \dt)\dt(\fraka)&=&(\id\ot \dt)\dt(a_1\dr P_r(\fraka'))\quad(\text{by Eq.}~(\mref{eq:drbase}))\\
&=&(\id\ot \dt)(\dt(a_1)\bul \dt( P_r(\fraka')))\quad(\text{by Proposition.}~\mref{prop:comuhom})\\
&=&(\id\ot \dt)\dt(a_1)\bul (\id\ot \dt)\dt( P_r(\fraka'))\\
&=&(\id\ot \dt)\dt(a_1)\bul (\id\ot \dt)(\id \ot P_r)\dt(\fraka')\quad(\text{by Eq.}~(\mref{eq:dtpr}))\\
&=&(\id\ot \dt)\dt(a_1)\bul (\id\ot \id\ot P_r)(\id \ot \dt)\dt(\fraka')\quad(\text{by Eq.}~(\mref{eq:comupr1}))\\
&=&(\dt\ot \id)\dt(a_1)\bul (\id\ot \id\ot P_r)(\dt \ot \id)\dt(\fraka')\ (\text{by the induction hypothesis}).
\end{eqnarray*}
On the other hand,
\begin{eqnarray*}
(\dt\ot \id)\dt(\fraka)&=&(\dt\ot\id)\dt(a_1\dr P_r(\fraka'))\quad(\text{by Eq.}~(\mref{eq:drbase}))\\
&=&(\dt\ot \id)(\dt(a_1)\bul \dt( P_r(\fraka')))\quad(\text{by Proposition.}~\mref{prop:comuhom})\\
&=&(\dt\ot \id)\dt(a_1)\bul (\dt\ot \id)\dt( P_r(\fraka'))\\
&=&(\dt\ot \id)\dt(a_1)\bul (\dt\ot \id)(\id \ot P_r)\dt(\fraka')\quad(\text{by Eq.}~(\mref{eq:dtpr}))\\
&=&(\dt\ot \id)\dt(a_1)\bul (\id\ot \id\ot P_r)(\dt\ot \id)\dt(\fraka')\quad(\text{by Eq.}~(\mref{eq:comupr2})).
\end{eqnarray*}
Thus $(\id\ot \dt)\dt(\fraka)=(\dt\ot \id)\dt(\fraka).$
This completes the induction.
\end{proof}

\begin{prop}
The $\bfk$-linear map $\vt$  satisfies the left counicity property:
\begin{equation}
(\vt\ot \id)\dt=\beta_\ell,
\end{equation}
where $\beta_\ell:\shan (A)\to \bfk\ot \shan (A),\fraka\mapsto 1\ot \fraka$ for all $\fraka\in A^{\ot k}, k\geq 1$.
\mlabel{prop:counit}
\end{prop}
\begin{proof}
We only need to verify
\begin{equation}
(\vt\ot \id)\dt(\fraka)=\beta_\ell(\fraka)
\mlabel{eq:vtl}
\end{equation}
for any pure tensor $\fraka\in A^{\ot k}$ with $k\geq 1$ for which we apply the induction on $k\geq 1$. If $k=1$, then $\fraka\in A$, and so Eq.~\eqref{eq:vtl} follows from the left counicity of $\vep_A$.

Let $k\geq 1$.
Assume that Eq.~(\mref{eq:vtl}) has been proved for $\fraka\in A^{\ot k}$. Consider $\fraka:=a_1\ot \fraka'\in A^{\ot (k+1)}$. Then
\begin{eqnarray*}
(\vt\ot \id )\dt(\fraka)&=&(\vt\ot \id)\dt(a_1\ot \fraka')\\
&=&(\vt\ot \id)\dt(a_1\dr P_r(\fraka'))\quad(\text{by Eq.}~(\mref{eq:drbase}))\\
&=&(\vt\ot \id)(\da(a_1)\bul\dt(P_r(\fraka')))\quad(\text{by Proposition}~\mref{prop:comuhom})\\
&=&(\vt\ot \id)\da(a_1)(\vt\ot \id)(\dt(P_r(\fraka')))\quad(\text{by Proposition}~\mref{prop:counithom})\\
&=&(\vt\ot \id)\da(a_1)(\vt\ot \id)(\id\ot P_r)\dt(\fraka')))\quad(\text{by Eq.}~(\mref{eq:dtpr}))\\
&=&(\vt\ot \id)\da(a_1)(\id\ot P_r)(\vt\ot \id)\dt(\fraka')))\\
&=&\beta_\ell(a_1)(\id\ot P_r)\beta_\ell(\fraka')\quad(\text{by the induction hypothesis})\\
&=&\beta_\ell(a_1)\beta_\ell(P_r(\fraka'))\quad(\text{by the definition of $\beta_\ell$})\\
&=&\beta_\ell(a_1\dr P_r(\fraka'))\quad(\text{by $\beta_\ell$ being an algebra isomorphism})\\
&=&\beta_\ell(\fraka).
\end{eqnarray*}
This completes the proof of Eq.~(\mref{eq:vtl}).
\end{proof}

Since $\vt$ is an extension of $\vep_A$, when $\vep_A$ does not satisfy the right counicity, nor will $\vt$. Assume that $\vep_A$ satisfies the right counicity, that is $A$ is a bialgebra. One can check that $\vt$ as defined is left counital, but is not right counital. For example, take $\fraka:=a_1\ot a_2\in A^{\ot 2}$. Define the linear map $\beta_r:\shan (A)\to \shan (A)\ot\bfk,\fraka\mapsto \fraka\ot 1$.  Then
\begin{eqnarray*}
(\id\ot \vt)\dt(\fraka)&=&(\id\ot\vt)\dt(a_1\ot a_2)\\
&=&(\id\ot\vt)\dt(a_1\dr P_r(a_2))\quad(\text{by Eq.}~(\mref{eq:drbase}))\\
&=&(\id\ot\vt)(\da(a_1)\bul\da(P_r(a_2)))\quad(\text{by Proposition}~\mref{prop:comuhom})\\
&=&(\id\ot\vt)\da(a_1)(\id\ot\vt)(\da(P_r(a_2)))\quad(\text{by Proposition}~\mref{prop:counithom})\\
&=&(\id\ot\vt)\da(a_1)(\id\ot\vt)(\id\ot P_r)\da(a_2)\quad(\text{by Eq.}~\eqref{eq:dtpr})\\
&=&(\id\ot\vep_A)\da(a_1)(\id\ot \vep_A)\da(a_2)\quad (\text{by Eq.~\eqref{eq:dfnvt}})\\
&=&\beta_r(a_1)\beta_r(a_2)\quad(\text{by the right counicity of $\vep_A$})\\
&=&\beta_r(a_1 \dr a_2)\quad(\text{by $\beta_r$ being an algebra isomorphism})\\
&=&\beta_r(a_1a_2)\\
&\neq&\beta_r(\fraka).
\end{eqnarray*}

Define
$$\mu_T:\bfk\to \shan (A),\,\,c\mapsto c1_A, c\in \bfk.$$
Then $\mu_T$ is a unit for  $(\shan (A),\dr)$.
Now let us put all the pieces together to give the main result of this section.
\begin{theorem}
The six-tuple $(\shan (A),\dr,\mu_T,\dt,\vt, P_r)$ is a left counital cocycle bialgebra.
\mlabel{thm:main}
\end{theorem}
\begin{proof}
First according to Theorem~\mref{thm:freeCNij}, the triple $(\shan (A),\dr, \mu_T, P_r)$ is a Nijenhuis algebra. Applying Proposition~\mref{prop:comult} and Proposition~\mref{prop:counit}, the triple $(\shan (A),\dt,\vt)$ is a left counital coalgebra. Then by Proposition~\mref{prop:comuhom} and Proposition~\mref{prop:counithom}, the quintuple $(\shan (A),\dr,\mu_T,\dt,\vt)$ is a left counital bialgebra satisfying the one-cocyle property in Eq.~(\mref{eq:cocy}), and hence the six-tuple $(\shan (A),\dr,\mu_T,\dt,\vt, P_r)$ is a left counital cocycle bialgebra.
\end{proof}

\section{The left counital Hopf algebra structure on free commutative Nijenhuis algebras}
\mlabel{sec:hopf}
In this section, we provide a left counital Hopf algebra structure on the free commutative Nijenhuis algebra $\shan (A)$.
\begin{defn}
\begin{enumerate}
\item
A left counital bialgebra $(H,m,\mu,\Delta,\vep)$ is called a {\bf graded left counital bialgebra} if there are $\bfk$-submodules $H^{(n)}$, $n\geq 0$, of $H$ such that
\begin{equation}
 H=\oplus_{n\geq 0} H^{(n)};\quad
H^{(p)}H^{(q)}\subseteq H^{(p+q)};\quad
\Delta(H^{(n)})\subseteq (H^{(0)}\ot H^{(n)})\oplus\Big(\bigoplus_{\substack{p+q=n\\p>0,q>0}} H^{(p)}\ot H^{(q)}\Big) \ \tforall p,q,n\geq 0.
\end{equation}
\item
A graded left counital bialgebra is called {\bf connected} if $H^{(0)}=\im \mu(=\bfk )$ and
\begin{equation}
\ker\vep=\bigoplus_{n\geq 1}H^{(n)}.
\mlabel{eq:kervep}
\end{equation}
\end{enumerate}
\end{defn}

Let $A$ be a $\bfk$-algebra. Let $C$ be a left counital $\bfk$-coalgebra. Denote $R:=\Hom(C,A)$. For $f,g\in R$, we can still define the {\bf convolution product} of $f$ and $g$ by
$$f\ast  g:=m_A(f\ot g)\Delta_C.$$

\begin{lemma}Let $A:=(A,m_A,\mu_A)$ be a $\bfk$-algebra. Let $C:=(C,\Delta_C,\vep_C)$ be a left counital $\bfk$-algebra. Let $e:=\mu_A\vep_C$.
Then $e$ is a left unit of $\bfk$-algebra $\Hom(C,A)$ under the convolution product $\ast$.
\end{lemma}
\begin{proof} Define a linear map $\alpha_\ell: \bfk\ot A\to A$, $k\ot a\mapsto ka$ for all $k\in \bfk$ and $a\in A$,  and a  linear map $\beta_\ell: C\to \bfk\ot C$, $c\mapsto 1_\bfk\ot c$ for all $c\in C$. Then $\alpha_\ell $ and $\beta_\ell$ are isomorphism. Since $\mu_A$ is the unit of $A$ and $\vep_C$ is the left counit of $C$, we have $m_A(\mu_A\ot \id_A)=\alpha_\ell$ and $(\vep_C\ot \id_C)\Delta_C=\beta_\ell$.
For every $f\in \Hom(C,A)$, we have
$$
e\ast f = m_A(\mu_A \vep_C\ot f)\Delta_C
= m_A(\mu_A\ot \id_A)(\id_\bfk\ot f)(\vep_C\ot \id_C)\Delta_C
=\alpha_\ell (\id_\bfk\ot f )\beta_\ell
=f.$$
\end{proof}
As usual, by the idempotency of $e$, the surjectivity of $\vep$ and the injectivity of $\mu$, we obtain
\begin{lemma}
Let $\bfk$ be a field. Let $H:=(H,m,\mu,\Delta,\vep)$ be a connected graded left counital $\bfk$-bialgebra. Then
\begin{equation}
H=\im \mu\oplus\ker \vep.
\end{equation}
\mlabel{lem:Hdir}
\end{lemma}

\begin{defn}
Let $H:=(H,m,\mu,\Delta,\vep)$ be a left counital $\bfk$-bialgebra. Let $e:=\mu\vep$.
\begin{enumerate}
\item
A $\bfk$-linear map $S$ of $H$ is called a {\bf right antipode} for $H$ if
\begin{equation}
 \id_H\ast S=e.
\end{equation}
\item
A left counital bialgebra with a right antipode is called a {\bf left counital right antipode Hopf algebra} or simply a {\bf left counital Hopf algebra}.
\end{enumerate}
\end{defn}

Since any Hopf algebra or any right Hopf algebra in the sense of~\mcite{GNT} is a left counital (right antipode) Hopf algebra, there are plenty examples of left counital Hopf algebras.

\begin{prop}
Let $\bfk$ be a field. Let $H:=(H,m,\mu,\Delta,\vep)$ be a connected graded left counital $\bfk$-bialgebra. Then for any $x\in H^{(n)},n\geq 1$,
\begin{equation}
\Delta(x)=1\ot x+\tilde{\Delta}(x), \quad\tilde{\Delta}(x)\in \ker\vep\ot \ker\vep.
\end{equation}
\end{prop}
\begin{proof}
According to $\Delta(H^{(n)})\subseteq  (H^{(0)}\ot H^{(n)})\oplus (\bigoplus_{\substack{p+q=n\\p>0,q>0}}H^{(p)}\ot H^{(q)})$ and $H^{(0)}=\bfk$, we get
$$\Delta(x)=1\ot u+\tilde{\Delta}(x),$$
where $u\in H^{(n)}$ and $\tilde{\Delta}(x)\in \bigoplus_{\substack{p+q=n\\p>0,q>0}}H^{(p)}\ot H^{(q)}$. Then by Eq.~(\mref{eq:kervep}), $\tilde{\Delta}(x)\in \ker\vep\ot \ker\vep$. By the left counit property, we obtain
$$
x=\beta_\ell^{-1}(\beta_\ell(x))
=\beta_\ell^{-1}(\vep\ot I)\Delta(x)
=\beta_\ell^{-1}(\vep\ot I)(1\ot u+\tilde{\Delta}(x))
=u.
$$
Thus,
$$\Delta(x)=1\ot x+\tilde{\Delta}(x).$$
\end{proof}

\begin{theorem}Let $\bfk$ be a field. Then
a connected graded  left counital  $\bfk$-bialgebra is a left counital Hopf algebra. Its antipode $S$ is defined by the recursion

\begin{equation}
S(1_H)=1_H, \ S(x) = -\sum_{(x)} x'S(x''), \ x\in \ker \vep,
\mlabel{eq:antipode}
\end{equation}
using  Sweedler's notation
$\tilde{\Delta}(x)=\sum_{(x)}x'\ot x''.$
\mlabel{thm:conn}
\end{theorem}
\begin{proof}
The proof is the same as the proof of the fact that a connected graded bialgebra is a Hopf algebra~\mcite{Gub,Ma}, by checking directly that the map $S$ defined by Eq.~\eqref{eq:antipode} satisfies
$$(\id_H\ast S)(x)=e(x)$$
for $x=1_H$ and $x\in \ker \vep$.
\end{proof}

\begin{theorem} Let $A=\oplus_{n\geq 0}A^{(n)}$ be a connected graded left counital bialgebra. Then the bialgebra $\shan (A)$ is a left counital Hopf algebra.
\mlabel{thm:Hopf}
\end{theorem}
\begin{proof}
By Theorem~\mref{thm:conn}, it suffices to show that $\shan (A)$ is a connected graded left counital bialgebra.
For any element $0\neq a\in A$, we define the degree of $a$ by defining
\begin{equation}
\deg(a):=k, \, \text{ if } a\in A^{(k)},k\geq 0.
\end{equation}
In general, for any pure tensor $0\neq \fraka:=a_1\ot a_2\ot \cdots \ot a_m\in A^{\ot m}$, we define
\begin{equation}
\deg(\fraka):=\deg(a_1)+\deg(a_2)+\cdots+\deg(a_m)+m-1.
\mlabel{eq:dfndg}
\end{equation}
Then for $m\geq 2$, taking $\fraka= a_1\ot \fraka'$ with $\fraka':=a_2\ot \cdots\ot a_m$, we have
\begin{equation}
\deg(\fraka)=\deg(a_1)+\deg(\fraka')+1.
\mlabel{eq:degree}
\end{equation}

Let $\fraku:=\shan (A)$. We denote the linear span of $\{\fraka\in\fraku\,|\,\deg(\fraka)=k\}$ by $\fraku^{(k)}$. Then we get $\fraku^{(0)}=A^{(0)},$ and thus $\fraku^{(0)}=\bfk.$ Furthermore, we have
$$A^{(n)}\subseteq \fraku^{(n)},\,\fraku^{(i)}\cap\fraku^{(j)}=\{0\},\,i\neq j,  \,\,\text{and so}\,\, \fraku=\oplus_{n\geq 0}\fraku^{(n)}.$$
Let $\fraka=a_1\ot\fraka'\in A^{\ot m}$. Then by Eq.~(\mref{eq:degree}),
\begin{equation}
\deg(\fraka)=n\Rightarrow \deg(\fraka')=n-\deg(a_1)-1.
\mlabel{eq:degco1}
\end{equation}
and
\begin{equation}
\deg(\fraka')=n \Rightarrow \deg(\fraka)=\deg(a_1)+n+1
\mlabel{eq:degco2}
\end{equation}
To prove that
\begin{equation}
\fraku^{(p)}\dr\fraku^{(q)}\subseteq \fraku^{(p+q)},\,\tforall p,q\geq 0,
\mlabel{eq:fraku}
\end{equation}
it suffices to prove
\begin{equation}
\deg(\fraka\dr \frakb)=\deg(\fraka)+\deg(\frakb),\tforall \fraka\in \fraku^{(p)},\frakb\in\fraku^{(q)}.
\mlabel{eq:deg}
\end{equation}
Use induction on $p+q\geq 0$. If $p=q=0$, then $\fraku^{(p)}=\fraku^{(q)}=\fraku^{(0)}=\bfk$, and thus $\fraka\dr\frakb\in\bfk$. Then $\deg(\fraka\dr\frakb)=0=\deg(\fraka)+\deg(\frakb)$.
Assume that Eq.~(\mref{eq:deg}) has been proved for $p+q\leq k$. Consider $p+q=k+1$. If either $p=0$ or $q=0$, then $\fraku^{(P)}=\bfk$ or $\fraku^{(q)}=\bfk$. Then by the definition of $\dr$ in Eq.~(\mref{eq:dfndr}), either
$\fraka\dr\frakb=\frakb$
or
$\fraka\dr\frakb=\fraka$.
So Eq.~(\mref{eq:deg}) holds.
Thus we assume $p,q\geq 1$. Let $\fraka\in \fraku^{(p)}$ and $\frakb\in \fraku^{(q)}$. If $\fraka, \frakb\in A$, then $\fraka\dr \frakb =\fraka \frakb$, and so $$\deg(\fraka\dr\frakb)=\deg(\fraka\frakb)=\deg(\fraka)+\deg(\frakb).$$
If $\fraka\in A$ and $\frakb:=b_1\ot \frakb'\in A^{\ot \ell},\ell\geq 2$, then $\fraka\dr \frakb=\fraka\frakb_1\ot\frakb'$. By Eq.~(\mref{eq:degree}, we get
\begin{eqnarray*}
\deg(\fraka\dr\frakb)&=&\deg(\fraka b_1\ot\frakb)\\
&=&\deg(\fraka b_1)+\deg(\frakb')+\ell-1\\
&=&\deg(\fraka)+\deg(b_1)+\deg(\frakb')+\ell-1\\
&=&\deg(\fraka)+\deg(\frakb).
\end{eqnarray*}
Similarly, Eq.~(\mref{eq:deg}) holds for $\fraka\in A^{\ot m}$, $m\geq 2$, and $ \frakb\in A$.
We now consider pure tensors $\fraka:=a_1\ot \fraka'\in A^{\ot \ell}$ and $\frakb:=b_1\ot \frakb'\in A^{\ot m}$ with $\ell,m\geq 2$. Then
\begin{eqnarray*}
\fraka\dr \frakb&=&(a_1\dr P_r(\fraka'))\dr (b_1\dr P_r(\frakb'))\\
&=&(a_1b_1)\dr P_r(\fraka'\dr P_r(\frakb')+P_r(\fraka')\dr\frakb'-P_r(\fraka'\dr \frakb'))\\
&=&(a_1b_1)\ot (\fraka'\dr P_r(\frakb')+P_r(\fraka')\dr\frakb'-P_r(\fraka'\dr \frakb')).
\end{eqnarray*}
By Eq.~(\mref{eq:degco1}), we obtain $$\deg(\fraka')=p-\deg(a_1)-1 \quad \text{and}\quad\deg(\frakb')=q-\deg(b_1)-1,$$
and then by Eq.~(\mref{eq:degco2}),
$$\deg(P_r(\fraka'))=\deg(1_A\ot\fraka')=p-\deg(a_1)\quad\text{and}\quad \deg(P_r(\frakb'))=\deg(1_A\ot \frakb')=q-\deg(b_1).$$
By the induction hypothesis,  we get
$$\deg(\fraka'\dr P_r(\frakb'))=\deg( P_r(\fraka')\dr\frakb')=\deg( P_r(\fraka'\dr \frakb'))= p+q-\deg(a_1)-\deg(b_1)-1.$$
Thus by Eq.~(\mref{eq:degree}),
\begin{eqnarray*}\deg(\fraka \dr \frakb)&=&\deg(a_1b_1)+\deg((\fraka'\dr P_r(\frakb')+P_r(\fraka')\dr\frakb'-P_r(\fraka'\dr \frakb'))+1\\
&= &\deg(a_1)+\deg(b_1)+p+q-\deg(a_1)-\deg(b_1)-1+1\\
&=&p+q\\
&=&\deg(\fraka)+\deg(\frakb).
\end{eqnarray*}
This completes the induction.

Finally, we prove
\begin{equation}
\dt(\fraku^{(n)})\subseteq (\fraku^{(0)}\ot \fraku^{(n)})\oplus\Bigg(\bigoplus_{\substack{p+q=n\\p>0,q>0}} \fraku^{(p)}\ot \fraku^{(q)}\Bigg)\ \tforall\,p,q,n\geq 0,
\mlabel{eq:sum}
\end{equation}
by induction on $n\geq0$. If $n=0$, then $\fraku^{(0)}=\bfk$, and so $\dt(\fraku^{(0)})=\da(\bfk)=\bfk\ot \bfk=\fraku^{(0)}\ot \fraku^{(0)}$. Assume that Eq.~(\mref{eq:sum}) has been proved for $n\geq 0$. Consider $\fraka\in \fraku^{(n+1)}$. If $\fraka\in A(=\bigoplus_{k\geq0}A^{(k)})$, then we get $\fraka\in A^{(n+1)}$. Since $A$ is a connected graded left counital bialgebra, we have
$$\dt(\fraka)=\da(\fraka)\subseteq (A^{(0)}\ot A^{(n+1)})\oplus\Bigg(\bigoplus_{\substack{p+q=n+1\\p>0,q>0}}A^{(p)}\ot A^{(q)}\Bigg)\subseteq (\fraku^{(0)}\ot \fraku^{(n+1)})\oplus \Bigg(\bigoplus_{\substack{p+q=n+1\\p>0,q>0}}\fraku^{(p)}\ot \fraku^{(q)}\Bigg).$$
Let $\fraka=a_1\ot \fraka'\in A^{\ot(\ell+1)}$ with $\ell\geq 1$. Then
\begin{eqnarray*}
\dt(\fraka)&=&\dt(a_1\dr P_r(\fraka')) \quad(\text{by Eq.~}(\mref{eq:drbase}))\\
&=&\dt(a_1)\bul\dt(P_r(\fraka'))\quad(\text{by Proposition}~\mref{prop:comuhom})\\
&=&\dt(a_1)\bul(\id\ot P_r)\dt(\fraka')\quad(\text{by Proposition}~\mref{eq:dtpr})\\
&=&\da(a_1)\bul(\id\ot P_r)\dt(\fraka').
\end{eqnarray*}
By Eq.~(\mref{eq:degco1}), we get $\fraka'\in \fraku^{(n-\deg(a_1))}$. Then by the induction hypothesis,
$$\dt(\fraka')\subseteq (\fraku^{(0)}\ot \fraku^{(n-\deg(a_1))})\oplus\Bigg(\bigoplus_{\substack{p_2+q_2=n-\deg(a_1)\\p_2>0,q_2>0}}\fraku^{(p_2)}\ot \fraku^{(q_2)}\Bigg).$$
Thus
\begin{eqnarray*}
\dt(\fraka)&=&\da(a_1)\bul(\id\ot P_r)\dt(\fraka')\\
&\subseteq&\Bigg(\Bigg(A^{(0)}\ot A^{(\deg(a_1))}\Bigg)\oplus\Bigg(\bigoplus_{\substack{p_1+q_1=\deg(a_1)\\p_1>0,q_1>0}}A^{(p_1)}\ot A^{(q_1)}\Bigg)\Bigg)\\
&&\bul(\id\ot P_r)\,\Bigg( \fraku^{(0)}\ot \fraku^{(n-\deg(a_1))}\oplus\Bigg(\bigoplus_{\substack{p_2+q_2=n-\deg(a_1)\\p_2>0,q_2>0}}\fraku^{(p_2)}\ot \fraku^{(q_2)}\Bigg)\Bigg)\\
&\subseteq&\Bigg( \Bigg(\fraku^{(0)}\ot \fraku^{(\deg(a_1))}\Bigg)\oplus\Bigg(\bigoplus_{\substack{p_1+q_1=\deg(a_1)\\p_1>0,q_1>0}}A^{(p_1)}\ot A^{(q_1)}\Bigg)\Bigg)\\
&&\bul\Bigg( (\fraku^{(0)}\ot \fraku^{(n-\deg(a_1)+1)})\oplus\Bigg(\bigoplus_{\substack{p_2+q_2=n-\deg(a_1)\\p_2>0,q_2>0}}\fraku^{(p_2)}\ot \fraku^{(q_2+1)}\Bigg)\Bigg)\\
&\subseteq&\Bigg(\fraku^{(0)}\ot \fraku^{(n+1)}\Bigg)\oplus\Bigg(\bigoplus_{\substack{p_2+q_2=n-\deg(a_1)\\p_2>0,q_2>0}}\fraku^{(p_2)}\ot \fraku^{(q_2+1+\deg(a_1))}\Bigg)\\
&&\oplus \Bigg(\bigoplus_{\substack{p_1+q_1=\deg(a_1)\\p_1>0,q_1>0}}\fraku^{(p_1)}\ot \fraku^{(q_1+n-\deg(a_1)+1)}\Bigg)\\
&&\oplus\Bigg( \Bigg(\bigoplus_{\substack{p_1+q_1=\deg(a_1)\\p_1>0,q_1>0}}\fraku^{(p_1)}\ot \fraku^{(q_1)}\Bigg)\,\,\bul \Bigg(\bigoplus_{\substack{p_2+q_2=n-\deg(a_1)\\p_2>0,q_2>0}}\fraku^{(p_2)}\ot \fraku^{(q_2+1)}\Bigg)\Bigg)\\
&\subseteq&(\fraku^{(0)}\ot \fraku^{(n+1)})\oplus\Bigg(\bigoplus_{\substack{p_1+q_1=\deg(a_1)\\p_2+q_2=n-\deg(a_1)\\p_1+p_2>0,q_1+q_2>0}} \fraku^{(p_1+p_2)}\ot\fraku^{(q_1+q_2+1)}\Bigg)\\
&\subseteq&\big(\fraku^{(0)}\ot \fraku^{(n+1)}\big)\oplus \Bigg(\bigoplus_{\substack{p+q=n+1\\p>0,q>0}} \fraku^{(p)}\ot\fraku^{(q)}\Bigg)\quad(p:=p_1+p_2,q:=q_1+q_2+1),\\
\end{eqnarray*}
as needed. This completes the proof.
\end{proof}

\smallskip

\noindent {\bf Acknowledgements}: This work was supported by the National Natural Science Foundation of
China (Grant No.~11371178 and 11601199).

\end{document}